\documentclass[10pt]{amsart}%
\usepackage[top=3cm, bottom=3cm, left=3cm, right=3cm] {geometry}
\usepackage{amssymb}
\usepackage{color}
\usepackage{amsmath}
\usepackage{graphicx}
\usepackage{amsfonts}%
\newcommand{\EE}{\ensuremath{\mathbb{E}}}

\newcommand{\NN}{\ensuremath{\mathbb{N}}}

\newcommand{\PP}{\ensuremath{\mathbb{P}}}

\newcommand{\RR}{\ensuremath{\mathbb{R}}}

\newcommand{\ZZ}{\ensuremath{\mathbb{Z}}}

\renewcommand{\phi}{\varphi}

\newcommand{\eps}{\ensuremath{\epsilon}}

\newcommand{\bB}{\ensuremath{\mathcal{B}}}

\newcommand{\fF}{\ensuremath{\mathcal{F}}}

\newcommand{\ltn}{\ensuremath{\left| \! \left| \! \left|}}
\newcommand{\rtn}{\ensuremath{\right| \! \right| \! \right|}}

\newtheorem{theorem}{Theorem}

\newtheorem{definition}[theorem]{Definition}

\newtheorem{lemma}[theorem]{Lemma}

\newtheorem{remark}[theorem]{Remark}

\title[Exponential stability of stochastic evolution equations driven by small fBm]{Exponential stability of stochastic evolution equations driven by small fractional Brownian motion with Hurst parameter in $(1/2,1)$}

\author{L.H. Duc}
\address[Luu Hoang Duc]{Max-Planck-Institut f\"ur Mathematik in den Naturwissenschaften, Inselstr. 22, 04103 Leipzig,
Germany}
\email[Luu Hoang Duc]{duc.luu@mis.mpg.de}

\author{M.J. Garrido-Atienza}
\address[Mar\'{\i}a J. Garrido-Atienza]{Dpto. Ecuaciones Diferenciales y An\'alisis Num\'erico\\
Universidad de Sevilla, Apdo. de Correos 1160, 41080-Sevilla,
Spain} \email[Mar\'{\i}a J. Garrido-Atienza]{mgarrido@us.es}

\author{A. Neuenkirch}
\address[Andreas Neuenkirch] {Universit\"at Mannheim, Institut f\"ur Mathematik, A5, 6, D-68131, Mannheim, Germany}\email[Andreas Neuenkirch]{neuenkirch@kiwi.math.uni-mannheim.de}

\author{B. Schmalfu{\ss }}
\address[Bj{\"o}rn Schmalfu{\ss }]{Institut f\"{u}r Stochastik\\
Friedrich Schiller Universit{\"a}t Jena, Ernst Abbe Platz 2, D-77043\\
Jena, Germany\\
}
\email[Bj{\"o}rn Schmalfu{\ss }]{bjoern.schmalfuss@uni-jena.de}



\begin{document}

\begin{abstract}
This paper addresses the exponential stability of the trivial solution of some types of evolution equations driven by H\"older continuous functions with H\"older index greater than $1/2$. The results can be applied to the case of equations whose noisy inputs are given by a fractional Brownian motion $B^H$ with covariance operator $Q$, provided that $H\in (1/2,1)$ and ${\rm tr}(Q)$ is sufficiently small.
\end{abstract}

\maketitle

\section{Introduction}

In this article we investigate the exponential stability of SPDEs driven by a fractional Brownian motion of the following type
\begin{equation}\label{r}
  du(t)=(Au(t)+F(u(t)))dt+G(u(t))dB^H(t), \quad t \geq 0, \qquad u(0)=u_0,
\end{equation}
defined on a separable Hilbert space $V$. In this equation, $A$ is the generator of an exponentially stable semigroup $S$ on $V$,
$F,\,G$ are nonlinear operators that will be introduced in the next section, $u_0\in V$ and $B^H$ is a $V$--valued fractional Brownian motion with Hurst parameter $H\in (1/2,1)$.\\


The study of the stability of stochastic ordinary and stochastic partial differential equations providing relevant information on the longtime behavior of the solution of such equations has already given birth to a huge number of works.  Just four decades ago this issue was considered when the noisy input was given by a Brownian motion $B^{1/2}$, and first analysis for Ito equations was deeply addressed in the pioneering monograph by Khasmiskii \cite{MR2894052}, where different kinds of stability were studied, like stability in probability, moment stability and almost sure exponential stability. All these methods are based on the fact that an Ito equation with sufficiently smooth coefficients generates a Markov semigroup or a Markov evolution family. Besides, one can find in the monograph by Mao \cite{MR1275834} stability results of more general stochastic differential equations. It is impossible to cite the big amount of papers that deal with stability analysis by using Lyapunov functions/functional or the Razumikhin-Lyapunov technique. To name only a few, some recent papers are \cite{CH}, \cite{kloedenwu1}, \cite{kloedenwu2} . As a result, the stability behavior of the solutions to Ito equations is quite well understood.\\

During the last two decades stochastic differential equations driven by other noises than the Brownian motions came in the center of interests due to their applications.
One class of noises is the fractional Brownian motion $B^H$ (fBm), which is a centered Gau{\ss} process with a covariance function determined by a parameter $H\in (0,1)$, known as the Hurst parameter. For $H\not=1/2$ this stochastic process does not have some fundamental properties of the Brownian motion, especially this process is not a Markov process. In this background, it seems quite natural to wonder whether the stability analysis carried out for Ito equations can be extended to equations with the non-Markov driving process $B^H$. In the articles \cite{H1, H2,H3,H4} the existence of adapted stationary solutions to dissipative finite-dimensional SDEs driven by fBm and their speed of convergence to the stationary state is studied, by extending Markovian notions as strong Feller property, invariant measure and adaptedness to the non-Markovian setting. A different approach has been developed recently in \cite{GNSch} for the particular case of fBm with Hurst index $H > 1/2$, where the integral is understood in a pathwise way based on fractional calculus techniques, and the  {\it local} exponential stability of the trivial solution is established. The meaning of {\it local} is that the initial condition $u_0$ must belong to a neighborhood of zero. The method in that paper is based on a cut--off argument, involving suitable tempered random variables. \\

If one wishes to go one step further and consider {\it global} exponential stability then stopping times should come into play. The reason is that the estimate of the norm of the solution depends on the magnitude of the driven process, therefore we will define a sequence of stopping times $(T_i)_{i\in \ZZ}$ such that the sequence of paths $(\theta_i B^H)_{i\in \ZZ}$ has bounded H\"older norm on the interval $[T_i,T_{i+1}]$ (here $\theta$ represents the Wiener shift flow). We shall check that, by choosing a fractional Brownian motion with a covariance operator with sufficiently small trace, the corresponding sequence of stopping times does not have any cluster point and how they help in order to handle systems like (\ref{r}). This smallness condition on the noise is necessary to cover the gap produced between the Lipschitz constants of the mappings $F$ and $G$ and the exponential growth of the semigroup $S$. It is worth to mentioning that, if in our main result the noise is neglected, then the sufficient condition ensuring exponential stability of the trivial solution is the same than the well-known sufficient condition in the deterministic setting. The method then seems to be consistent and it leads to what seems to us as the first result on exponential stability results for solutions of (\ref{r}) when $H\in (1/2,1)$ by using the stopping times argument together with a suitable discrete Gronwall Lemma.\\

We would like to mention that some advances has been also obtained in the study of the asymptotic behavior of the solution to (\ref{r}) by using the theory of Random Dynamical Systems. One advantage of defining pathwise stochastic integrals is that they do not produce exceptional sets that could destroy the generation of a cocycle in the Hilbert-valued setting, as it happens with the Ito integral. Therefore, one then can consider the random dynamical system associated to equation (\ref{r}) and establish the existence of random attractors that pullback attract any solution, or the existence of other interesting objects like stable/unstable/invariant manifolds, that also give information of the longtime behavior of solutions. The reader is referred to the papers \cite{CGGSch14}, \cite{MR3226746}, and \cite{GLSch}, and the references therein.\\

The structure of the paper is the following: in Section \ref{s1} we give the different assumptions of the terms in (\ref{r}), then we introduce the pathwise integral using fractional calculus techniques and establish the existence of a unique solution to (\ref{r}). Section \ref{s2} concerns the analysis of the exponential stability of the trivial solution when the driving signal is a H\"older function, for which we construct a sequence of stopping times satisfying suitable conditions. In Section \ref{s3} we prove that all these nice properties of the stopping times hold true when taking $B^H$ as noise (for $H\in (1/2,1)$), for which its covariance should satisfied a smallness condition. The paper ends with a short section emphasizing that, if we perturb an exponentially stable deterministic system with a {\it small} fBm, then the resulting stochastic system is also exponentially stable.

We refer to \cite{DGSch17} for a short and recent announcement of our results.

\section{Analytic Preliminaries. Existence and uniqueness of mild solutions}\label{s1}

Given the separable Hilbert space ($V$, $\|\cdot\|$, $(\cdot,\cdot)_V$), our first goal is to give a definition of solution to the following $V$-valued SPDEs driven by a fractional Brownian motion
\begin{equation}\label{eq1}
  du(t)=(Au(t)+F(u(t)))dt+G(u(t))dB^H(t), \quad t \geq 0, \qquad u(0)=u_0\in V.
\end{equation}
\smallskip

Denote by $(L_2(V),\|\cdot\|_{L_2(V)})$ the separable Hilbert space of Hilbert--Schmidt operators.  If $(e_i)_{i \in \NN}$ forms a complete orthonormal basis of $V$, then the norm $\|\cdot\|_{L_2(V)}$ is defined by
$$\|z\|_{L_2(V)}^2=\sum_{i\in \NN} \|z e_i\|^2,$$
for $z\in L_2(V)$.

We will work under the following assumptions for the operator $A$ and the nonlinearities $F$ and $G$:
\begin{itemize}
 \item[{\bf(A1)}] $A$ is a strictly negative and symmetric operator with a compact inverse, that generates an analytic exponential stable semigroup $S$ on $V$.
\item[{\bf(A2)}] The mapping
$F:V\to V$ is globally Lipschitz--continuous with Lipschitz constant denoted by $c_{DF}$.
\item[{\bf(A3)}] The mapping  $G: V\to L_2(V)$ is a twice continuously
Fr\'echet--differentiable operator with bounded first and second
derivatives. We denote by $c_{DG}$ and $c_{D^2G}$, respectively, the bounds of $DG$ and $D^2G$.
\end{itemize}
Furthermore, $B^H$ is a $V$--valued fractional Brownian motion with Hurst parameter $H\in (1/2,1)$, see details in Section 4.

We shall interpret \eqref{eq1} in its mild form
\begin{equation}\label{eq2}
  u(t)=S(t)u_0+\int_0^tS(t-r)F(u(r))dr+\int_0^tS(t-r)G(u(r))dB^H(r), \quad t \geq 0,
\end{equation}
where the integral against $B^H$ is a Young integral with values in $V$, and hence a pathwise defined integral. Specifically, interpreting $B^H$ to be a canonical fractional Brownian motion $B^H(t,\omega)=\omega(t)$ for $t\in\RR$, we know that almost surely this canonical process is $\beta^\prime$--H{\"o}lder--continuous for any $\beta^\prime<H$, see Bauer \cite{Bau96}. Thus, we can reformulate equation \eqref{eq2} as
\begin{equation}\label{eq3}
  u(t)=S(t)u_0+\int_0^tS(t-r)F(u(r))dr+\int_0^tS(t-r)G(u(r))d\omega(r), \qquad t \geq 0.
\end{equation}

In this section we will briefly present the main tools to analyze the existence of solutions to equation \eqref{eq3}.
\medskip

\medskip

In order to define the integral with respect to $\omega$, we introduce the so--called fractional derivatives. More precisely, we define the right hand side fractional derivative of order $\alpha\in (0,1)$ of a sufficiently regular function $g$ and the left hand side fractional derivative of order $1-\alpha$ of of a sufficiently regular function $\omega_{t-}(\cdot):=\omega(\cdot)-\omega(t)$, given by the expressions
 \begin{align*}\label{fractder}
 \begin{split}
    D_{{s}+}^\alpha g[r]=&\frac{1}{\Gamma(1-\alpha)}\bigg(\frac{g(r)}{(r-s)^\alpha}+\alpha\int_{s}^r\frac{g(r)-g(q)}{(r-q)^{1+\alpha}}dq\bigg),\\
    D_{{t}-}^{1-\alpha} \omega_{{t}-}[r]=&\frac{(-1)^{1-\alpha}}{\Gamma(\alpha)}
    \bigg(\frac{\omega(r)-\omega(t)}{(t-r)^{1-\alpha}}+(1-\alpha)\int_r^{t}\frac{\omega(r)-\omega(q)}{(q-r)^{2-\alpha}}dq\bigg),
\end{split}
\end{align*}
where $0\le s\le r\le t$ and $\Gamma(\cdot)$ denotes the Gamma function, see \cite{Samko} for a comprehensive introduction of fractional derivatives.

In what follows, $(e_i)_{i\in\NN}$ will be the complete orthonormal base in $V$ generated by the eigenelements of $-A$ with associated eigenvalues $(\lambda_i)_{i\in\NN}$.

Assume $1-\beta^\prime <\alpha<\beta$ and $g\in C^\beta_\beta([0,T]; L_2(V))$ (see the definition of this space below), $\omega \in C^{\beta^\prime} ([0,T];V)$, such that
$$r\mapsto\|D_{s+}^\alpha g[r]\|_{L_2(V)}\|D_{t-}^{1-\alpha}\omega[r]\|$$ is Lebesgue integrable. Then for $0\le s\le r\le t \le T$ we can define
\begin{equation} \label{integral}
  \int_s^t g(r)d\omega(r):=(-1)^\alpha\sum_{j\in \NN}\bigg(\sum_{i\in \NN}\int_s^t D_{s+}^\alpha(e_j,g(\cdot)e_i)_{V}[r]D_{t-}^{1-\alpha}(e_i,\omega(\cdot))_{ V}[r]dr\bigg)e_j.
\end{equation}

This integral is well--defined and it is given by a generalization of the pathwise integral introduced by Z\"ahle \cite{Zah98}, which was given as an extension to a stochastic setting of the Young integral (see \cite {You36}). Also, for our further purposes, we will need to consider a mapping $\theta$ defined as
\begin{align}\label{shift}
\theta_t \omega (\cdot)=\omega(t+\cdot)-\omega(t),\quad t\in \RR.
\end{align}
This mapping, known as the Wiener shift, will be introduced with more details in Section \ref{s3} when considering the fractional Brownian motion as integrator. One property that we will use in the study of the stability of our problem refers to the behavior of the integral when performing a change of variable, which reads as follows
\begin{align}\label{change}
\int_s^t g(r)d\omega(r)= \int_{s-\tau}^{t-\tau} g(r+\tau)d\theta_\tau \omega(r),
\end{align}
see \cite{CGGSch14} for the proof.

\medskip

On the other hand, as a consequence of {\bf(A1)}, we can consider the fractional powers of $-A$ and introduce the spaces $V_\delta:=D((-A)^\delta)$ with norm $\|\cdot\|_{V_\delta}=\|(-A)^\delta\cdot\|$ for $\delta\ge 0$ such that $V=V_0$, see \cite{CGGSch14}.
Thanks to the analyticity of the semigroup, there exists a constant $c_S>0$ such that
\begin{align}
  \|S(t)\|_{L(V_\zeta, V_{\gamma})}&= \|(-A)^\gamma S(t)\|_{L(V_\zeta,V)}\le
  c_{S}t^{\zeta-\gamma}e^{-\lambda t}\qquad\text{for }
  \gamma \geq \zeta \geq 0
 \label{eq4},
  \end{align}
  \begin{align}
 \|S(t)-{\rm id}\|_{L(V_{\sigma},V_{\eta})} &\le c_S
t^{\sigma-\eta}, \quad \text{for }\eta\geq 0,\quad \sigma\in
[\eta,1+\eta]\label{eq5},
\end{align}
where $0<\lambda\le \lambda_1$.
The constant $c_S$ may depend on the time interval $[0,T]\ni t$, see Chueshov \cite{Chu02} Page 83. But choosing $\lambda<\lambda_1$, these estimates are true for all $t>0$ with $c_S$ depending on $\lambda_1-\lambda$ but not on $t$ (note that $c_S$ may depend on other parameters $\gamma,\,\sigma,\,\eta$, but we suppress this dependence in its notation. Also throughout the paper, the value of $c_S$ can change from line to line).

As usual, denote by $L(V_\sigma,V_\eta)$ (respectively $L(V)$) the space of continuous linear operators from $V_\sigma$ into $V_\eta$ (from $V$ into itself). From (\ref{eq4}) and (\ref{eq5}), for $0\leq q\leq r\leq s\leq t$, we can easily derive that
\begin{align}\label{eq30}
\begin{split}
& \|S(t-r)-S(t-q)\|_{L(V_{\delta},V_{\gamma})}\le c_S(r-q)^\eta (t-r)^{-\eta-\gamma+\delta},\,\, \text{for }
  \gamma \geq \delta-\eta \geq 0,
\\
 & \|S(t-r)- S(s-r)-S(t-q)+S(s-q)\|_{L(V)}\\
&\qquad \leq
c_S(t-s)^{\eta}(r-q)^{\gamma}(s-r)^{-(\eta+\gamma)},\,\, \text{for }
  \gamma, \, \eta \geq 0.
\end{split}
\end{align}

Since the properties on the semigroup does not ensure H\"older continuity at zero (see Lemma \ref{lem_damp_initial}), we will work with a damped H\"older norm:
\begin{equation*}
    \|u\|_{\beta,\beta}=\|u\|_{\beta,\beta,T_1,T_2}=\|u\|_{\infty, T_1,T_2}+\ltn u \rtn_{\beta,\beta,T_1,T_2},
\end{equation*}
where
\begin{equation*}
  \ltn u \rtn_{\beta,\beta,T_1,T_2}=\sup_{T_1< s<t\le T_2}(s-T_1)^\beta\frac{\|u(t)-u(s)\|}{(t-s)^\beta}
\end{equation*}
and denote by $C^\beta_\beta([T_1,T_2];V)$ the set of functions $u \in C([T_1,T_2];V)$ such that $\|u\|_{\beta,\beta} < \infty. $
It is known that $C^\beta_\beta([T_1,T_2];V)$  is a Banach space, see \cite{lunardi} and \cite{CGGSch14}.

The following lemma shows why we consider the norm $\|\cdot\|_{\beta,\beta}$ instead of the usual H\"older norm $\|\cdot\|_\beta$:

\begin{lemma}\label{lem_damp_initial} Let $0 \leq t_1 < t_2 < \infty$, $v \in V$, $\beta \in (0,1)$ and $\lambda <\lambda_1$.  Then we have
\begin{align*}
  \|S(\cdot)v\|_{\beta,\beta,t_1,t_2}
 &\le  c_S e^{-\lambda t_1}\|v\|.
\end{align*} 	
\end{lemma}
\begin{proof}
According to (\ref{eq4}) and (\ref{eq5}), we have that
\begin{align*}
  \|S(\cdot)v\|_{\beta,\beta, t_1,t_2}&= \sup_{t\in [t_1,t_2]}\|S(t)v\|+\sup_{t_1< s<t\le t_2}(s-t_1)^\beta\frac{\|S(t)v-S(s)v\|}{(t-s)^\beta}\\
  &\le  c_S e^{-\lambda t_1}\|v\|+c_S \sup_{t_1< s<t\le t_2}(s-t_1)^\beta\frac{ s^{-\beta} e^{-\lambda s} (t-s)^\beta }{(t-s)^\beta}\|v\|\\
 &\le  c_S  e^{-\lambda t_1}\|v\|, \end{align*}
since $s^{-\beta}\leq (s-t_1)^{-\beta}$.

However, for the usual H\"older norm $\|\cdot\|_\beta$ we obtain
\begin{align*}
  \|S(\cdot)v\|_{\beta, t_1,t_2}&\le  c_S e^{-\lambda t_1}\|v\|+c_S \sup_{t_1\le s<t\le t_2}\frac{ s^{-\beta} e^{-\lambda s} (t-s)^\beta }{(t-s)^\beta}\|v\|
 \end{align*}
 and the last term on the right hand side is infinite when considering $t_1=0$.
\end{proof}

On the other hand, from {\bf (A2)} we directly obtain, with $c_F:=\|F(0)\|$, that
\begin{equation}\label{condF}
     \|F(u)-F(v)\|\le c_{DF}\|u-v\|,\qquad  \|F(u)\|\le c_F+c_{DF}\|u\|,
\end{equation}
and thanks to {\bf(A3)}, if  $c_G:=\|G(0)\|_{L_2(V)}$, for $u_1,\,u_2,\,v_1,\,v_2\in V$, we have
\begin{align} \label{condG}
    &\|G(u_1)\|_{L_2(V)}\le c_G+ c_{DG}\|u_1\|,  \nonumber\\
    &\|G(u_1)-G(v_1)\|_{L_2(V)}\le c_{DG}\|u_1-v_1\|, \\ \nonumber
    &\|G(u_1)-G(v_1)-(G(u_2)-G(v_2))\|_{L_2(V)}\\ \nonumber
    & \quad \le c_{DG}\|u_1-v_1-(u_2-v_2)\|+c_{D^2G} \|u_1-u_2\|(\|u_1-v_1\|+\|u_2-v_2\|).
\end{align}

The following result will be crucial for our analysis, and the proof can be found in \cite{CGGSch14}.
\begin{lemma} \label{lem_cruc}
Let $T>0$, $\omega\in C^{\beta^\prime}([0,T];V)$, $1/2<\beta<\beta^\prime$, $1-\beta' < \alpha < \beta$ and  $u\in C^\beta_\beta([0,T];V)$.  Under assumptions {\bf(A1)-(A3)} and $G(0)=F(0)=0$
we have that
\begin{equation*}
   t\mapsto \int_0^tS(t-r)G(u(r))d\omega \, \in C^\beta_\beta([0,T];V)
\end{equation*}
where
\begin{align}\label{3_lemma}
\left\| \int_0^\cdot S(\cdot-r)G(u(r))d\omega \right \|_{\beta,\beta,0,T}\leq T^{\beta^\prime}c_{\alpha,\beta,\beta^\prime}c_S c_{DG}\ltn\omega\rtn_{\beta^\prime,0,T}\|u\|_{\beta,\beta,0,T}.
\end{align} Moreover, we have
\begin{equation*}
  t\mapsto \int_0^tS(t-r)F(u(r))dr \, \in C^\beta_\beta([0,T];V)
\end{equation*}
with
\begin{equation}\label{4_lemma}
\left \|\int_0^{\cdot}S(\cdot-r)F(u(r))dr \right \|_{\beta,\beta,0,T} \leq  T c_S c_{DF}\|u\|_{\infty,0,T}.
\end{equation}
Here $c_{\alpha, \beta, \beta'}>0$ denotes a constant which only depends on $\alpha$, $\beta$ and $\beta^\prime$.	
\end{lemma}

We would like to point out that in the previous lemma we have assumed that $G(0)=F(0)=0$, since we will make this simplification later when analyzing the exponential stability of the solutions. Needless to say that the regularity of the stochastic and non-stochastic integrals is the same when these two assumptions are omitted.\\

On account of (\ref{eq4})-(\ref{condG}), by means of the application of the Banach fixed point theorem we obtain the following result:
\begin{theorem}\label{t2}
Let $T>0$, $\omega\in C^{\beta^\prime}([0,T];V)$, $1/2<\beta<\beta^\prime$ and $u_0\in V$.  Under assumptions {\bf(A1)-(A3)}, there exists a unique solution $u\in C^\beta_\beta([0,T];V)$ of  \eqref{eq3}.
\end{theorem}

This result has been proved in \cite{CGGSch14}, with the help of an equivalent norm to $\|\cdot\|_{\beta,\beta}$ that depends on a weight. However, in that paper the authors do not consider any drift $F$ in the problem, although the drift is the harmless term and suitable estimates for establishing the existence could be derived thanks to its Lipschitz continuity (see  Lemma \ref{lem_cruc}).

\medskip

We finish this section introducing a discrete Gronwall-like lemma that we will use in the next section, and whose proof can be found in \cite{GNSch}.

\begin{lemma}\label{l9}
Let $(y_n)_{n \geq 0}$ and $(g_n)_{n \geq 0}$ be nonnegative sequences and $c>0$ a nonnegative constant. If\footnotemark
\footnotetext{The sum $\sum _{j=0}^{n-1}$ is assumed to be zero for $n=0$, while the product $\prod_{j=0}^{n-1}$ is assumed to be one for $n=0$.}
$$y_n\leq c+  \sum_{j=0}^{n-1}  g_j y_j, \qquad n=0,1,\ldots, $$
then
$$y_n \leq c \prod_{j=0}^{n-1}  (1+g_j) \qquad n=0,1, \dots .$$
\end{lemma}

\section{Exponential Stability of the trivial Solution}\label{s2}

The aim of this section is to prove that the trivial solution of (\ref{eq3}) is exponentially stable.

Let us denote by $u=u(t)=u_{u_0}(t,\omega)$ the dynamical system generated by the unique global solution of \eqref{eq3}, see Chen {\it et al.} \cite{CGGSch14}. For the general definition
of a random dynamical system we refer to Arnold \cite{Arn98}.\\

\begin{definition}\label{d1}
We say that the the dynamical system $u$ generated by the unique global solution of \eqref{eq3} is
exponentially  stable with respect to the steady state zero and with exponential rate $\rho>0$ (or equivalently, that the trivial solution of (\ref{eq3}) is exponential stable with rate $\rho>0$), if for any bounded set of initial conditions $u(0)\in V$ there exists a random variable $C(\omega)\ge 0$ such that
$$ \|u(t)\|\le C(\omega)e^{-\rho t}$$
for $t\ge0$ and $\omega\in\Omega$.
\end{definition}

From now on, we assume that (\ref{eq3}) possesses the trivial solution, which means that we assume that $F(0)=0$ and $G(0)=0$.
\smallskip

The technique to derive the exponential stability will be based on the construction of an increasing sequence of  (stopping) times $(T_i)_{i \in \mathbb{Z}}$ in such a way that in any interval $[T_i,T_{i+1}]$ the norm of the noise will be small enough.

Given $t\in [T_n,T_{n+1}]$, thanks to the additivity of the integrals and (\ref{change}), we can consider the following splitting of the solution:
\begin{align}
\begin{split}\label{sp}
  u(t) = & S(t)u_0+\sum_{i=0}^{n-1}\int_{T_i}^{T_{i+1}}S(t-r)F(u(r))dr+\sum_{i=0}^{n-1}\int_{T_i}^{T_{i+1}}S(t-r)G(u(r))d\omega(r)\\
  &+\int_{T_n}^{t}S(t-r)F(u(r))dr+\int_{T_n}^tS(t-r)G(u(r))d\omega(r) \\
  =& S(t)u_0\\
  &+\sum_{i=0}^{n-1}S(t-T_{i+1})\int_{0}^{T_{i+1}-T_i}S(T_{i+1}-T_i-r)F(u(r+T_i))dr\\
  &+\sum_{i=0}^{n-1}S(t-T_{i+1})\int_{0}^{T_{i+1}-T_i}S(T_{i+1}-T_i-r)G(u(r+T_i))d\theta_{T_i}\omega(r)\\
  &+\int_0^{t-T_n}S(t-T_n-r)F(u(r+T_n))dr\\
  &+\int_0^{t-T_n}S(t-T_n-r)G(u(r+T_n))d\theta_{T_n}\omega(r).
  \end{split}
\end{align}
We want to estimate $\|u(t)\|_{\beta,\beta,T_n,T_{n+1}}$. To simplify the presentation, let us abbreviate $u(\cdot+T_i)$ by $u^i(\cdot)$, that is,
$$u(\tau+T_i)=u^i(\tau), \qquad  \tau \in [0,T_{i+1}-T_i].$$
In particular we have
$$u(t)=u^i (t-T_i),\qquad t-T_i\in [0, T_{i+1}-T_i],$$ and
therefore, $$ \|u\|_{\beta,\beta,T_i,T_{i+1}}= \|u^i \|_{\beta,\beta,0,T_{i+1}-T_i}.$$
In the following, we will use the abbreviation $$ \|u^i\|_{\beta,\beta}= \|u^i \|_{\beta,\beta,0,T_{i+1}-T_i}, \qquad i=0,1,2, \ldots $$

We need the following assumptions:
\begin{itemize}
\item[{\bf(S1)}]
For every $\mu \in (0,1/{c_S c_{DF}})$  there exists an increasing sequence $(T_{i})_{i \in \mathbb{Z}}$ such that $T_0=0$ and
$$ 	c_{\alpha,\beta,\beta^\prime}  c_{DG} \ltn\theta_{T_i}\omega\rtn_{\beta^\prime,0,T_{i+1}-T_i}(T_{i+1}-T_i)^{\beta^\prime} +  c_{DF} (T_{i+1}-T_i) = c_{DF} \mu, \qquad i \in \mathbb{Z}.$$
\item[{\bf(S2)}] For every $\mu \in (0,1/{c_S c_{DF}})$, there exists a $D=D(\mu,\omega) \in (0,\mu]$  such that
$$ \liminf_{k \rightarrow \infty} \frac{T_{k}}{k} \ge D. $$
\item[{\bf(S3)}] There exist $\mu \in (0,1/{c_S c_{DF}})$ such that
 \begin{align} \label{rho} \rho^*:=\lambda- \frac{c_S c_{DF} \mu}{(1- c_S c_{DF} \mu)D}e^{{\lambda \mu}} >0.
 \end{align}
\end{itemize}

\begin{remark}
As we will see in the main result of that paper (Theorem \ref{es} below) the previous assumptions are needed to establish the exponential stability of the solution of \eqref{eq3} with order $\rho<\rho^*$, with $\rho^*$ defined by (\ref{rho}). In fact, {\bf(S1)} is assumed to deal with the contribution of the Young integral with respect to $\omega$, while that {\bf(S2)} means a lower bound on the growth of the times $T_k$, where $D$ depends on $c_{DF},\,c_{DG},\,\mu$ and on the asymptotic behavior of $\omega$. In the following we will suppress the dependence on $c_{DF},\,c_{DG}$ in this notation. Later $\omega$ will be a canonical fractional Brownian motion which is ergodic. Then the asymptotical behavior will be given in terms of the trace of the covariance operator of $\omega$, denoted by ${\rm tr}(Q)$.\\

We also would like to stress that {\bf(S1)} ensures that $T_{i+1}- T_i\in (0,\mu]$, for any $i\in \ZZ$, and in particular $T_1\in (0,\mu]$.
There are other possibilities to define the stopping times $T_i$.
This definition  allows us to make a comparison between the stochastic system and its corresponding deterministic system if $F\not\equiv0$, see Section \ref{s4} below.
\end{remark}

Assuming these assumptions, we are in conditions to prove the main result of this paper.
 \begin{theorem}\label{es}
Let $T>0$, $\omega\in C^{\beta^\prime}([0,T];V)$ with $\beta' >1/2$ and let $0<\lambda < \lambda_1$, where $\lambda_1$ denotes the smallest eigenvalue of $-A$.
Moreover, assume {\bf(A1)}-{\bf (A3)},{\bf(S1)}-{\bf(S3)} and $F(0)=G(0)=0$. Then the solution of \eqref{eq3} satisfies\begin{equation} \label{est}
  \lim_{t\to\infty}\|u(t)\|e^{ \rho t}=0
\end{equation}
for all $0<\rho<\rho^*$, where $\rho^*$ is given by (\ref{rho}) and the above convergence is uniform with respect to bounded sets of initial conditions.
\end{theorem}

\begin{proof}
By Lemma \ref{lem_damp_initial} and Lemma \ref{lem_cruc}, on account of the assumption {\bf (S1)}, from (\ref{sp}) and for $n=0,\,1,\ldots$ we obtain
\begin{align}\label{eq16}
\begin{split}
  \|u^n&\|_{\beta,\beta}\le c_S e^{-\lambda T_n }\|u_0\|\\
  &+c_S c_{DF} \sum_{i=0}^{n-1}e^{-\lambda(T_n-T_{i+1})} (T_{i+1}-T_i) \|u^i\|_{\beta,\beta} + c_S c_{DF} (T_{n+1} -T_n)\|u^n\|_{\beta,\beta}
    \\
& + c_{\alpha,\beta,\beta^\prime} c_S c_{DG} \sum_{i=0}^{n-1}e^{-\lambda(T_n -T_{i+1})} \ltn\theta_{T_i}\omega\rtn_{\beta^\prime,0,T_{i+1}-T_i}(T_{i+1}-T_i)^{\beta^\prime} \|u^i\|_{\beta,\beta}\\ & +
  c_{\alpha,\beta,\beta^\prime} c_S c_{DG} \ltn\theta_{T_n}\omega\rtn_{\beta^\prime,0,T_{n+1}-T_n}(T_{n+1}-T_n)^{\beta^\prime}\|u^n\|_{\beta,\beta}\\
&\le c_S e^{-\lambda T_n }\|u_0\|+ c_S c_{DF} \mu \sum_{i=0}^{n-1}e^{-\lambda(T_n-T_{i+1})}  \|u^i\|_{\beta} + c_S c_{DF} \mu \|u^n\|_{\beta,\beta}.
\end{split}
\end{align}
Multiplying both sides with $e^{\lambda T_n }$, setting $y_n= e^{\lambda T_n } \|u_{n}\|_{\beta,\beta}$, $n =0,1, \ldots$, gives
$$ y_n \leq \frac{c_S}{1-c_S c_{DF} \mu} \|u_0\| + \frac{c_S c_{DF} \mu}{1- c_S c_{DF} \mu} \sum_{i=0}^{n-1} e^{\lambda (T_{i+1}-T_i)} y_i, \qquad n=0,1, \ldots .$$
Therefore, applying Lemma \ref{l9} with
$$ c=  \frac{c_S}{1-c_S c_{DF} \mu} \|u_0\|, \qquad g_j= \frac{c_S c_{DF} \mu}{1- c_S c_{DF}\mu} e^{\lambda (T_{j+1}-T_j)}, \quad j=0,1, \ldots, $$
yields
\begin{align*}
\| u^n \|_{\beta,\beta} \leq \frac{c_S}{1- c_S c_{DF} \mu} \|u_0\|  \prod_{j=0}^{n-1} \left( 1 + \frac{c_S c_{DF} \mu}{1-c_S c_{DF}\mu}e^{\lambda (T_{j+1}-T_j)} \right) e^{-\lambda T_n}, \qquad n =0,1, \ldots.
\end{align*}

Note that $\bf{(S1)}$ trivially implies that
$$ T_{i+1}- T_i \leq \mu, \qquad i=0,1, \ldots . $$
Hence we obtain
\begin{align} \label{next_ste}
\begin{split}
\| u^n \|_{\beta,\beta} &\leq \frac{c_S}{1-c_S c_{DF} \mu} \|u_0\|  e^{- \lambda T_n}  \left( 1 + \frac{c_S c_{DF} \mu}{1- c_S c_{DF} \mu}e^{{\lambda \mu}} \right)^n\\&\leq \frac{c_S}{1- c_S c_{DF} \mu} \|u_0\|  e^{- \lambda T_n + \frac{c_S c_{DF} \mu}{1- c_S c_{DF}\mu}e^{\lambda \mu} n}, \quad n =0,1, \ldots.
\end{split}
\end{align}
 It remains to analyze the exponent of the last expression more precisely. Thanks to assumption {\bf(S2)} there exists $\eps>0$ and $n_{D(\mu, \eps,  \omega)}\in \NN$ such that
$$ \frac{T_n}{D-\eps} \geq  n, \qquad n \geq n_{D(\mu,\eps,\omega)}. $$
Using this estimate we have
 $$ - \lambda T_n + \frac{c_S c_{DF} \mu}{1- c_S c_{DF}\mu}e^{\lambda \mu} n \leq \left(-\lambda+ \frac{ c_S c_{DF}\mu }{(1- c_S c_{DF}\mu)(D-\eps)}e^{\lambda \mu} \right)T_n=-\rho T_n, \qquad n \geq n_{D(\mu,\eps, \omega)},$$
where $\eps>0$ is chosen so small that $0<\rho<\rho^\ast$. Then replacing this inequality in (\ref{next_ste}) we obtain
 \begin{align}\label{e1}
\| u^n \|_{\beta,\beta} \leq \frac{c_S}{1 - c_S c_{DF}\mu} \|u_0\| e^{ -\rho T_n }, \qquad n \geq n_{D(\mu,\eps, \omega)}.
 \end{align}
Given any $t > T_{n_{D(\mu,\eps,\omega)}}(\omega)$ there exists $n(t) > n_{D(\mu,\eps,\omega)}$ such that $t\in [T_{n(t)}(\omega),T_{n(t)+1}(\omega))$, that is, $T_{n(t)}(\omega) \geq t-\mu$,  and from (\ref{e1}) and {\bf(S3)} we can derive that
 \begin{align*}
  \| u(t) \| &\leq \frac{c_S}{1 - c_S c_{DF} \mu} \|u_0\| e^{ -\rho (t-\mu)}.
 \end{align*}
The conclusion follows  because the above calculations are true for any $\eps>0$.
\end{proof}

From the conclusion of the last theorem we obtain exponential stability of the dynamical system $u$. Indeed, suppose that $\|u_0\|\le R$, then we obtain
\begin{equation*}
  \|u(t)\|\le C(\omega)e^{-\rho t}
\end{equation*}
for $t\ge 0$ where
\begin{equation*}
  C(\omega)=C(\omega)=\frac{c_S}{1 - c_S c_{DF} \mu}R  e^{\rho\mu}+ e^{\rho T_{n_0(D,\eps,\omega)}} \sup_{\|u_0\|\le R}\sup_{t\le T_{n_0(D,\eps,\omega)}} \|u_{u_0}(t,\omega)\|
\end{equation*}
where $T_{n_0(D,\eps,\mu)}$ is defined in the proof of the last theorem. The finiteness of $C(\omega)$ follows by Chen {\it et al.} \cite{CGGSch14}.

\section{Stopping time analysis}\label{s3}

In this section, our main goal is to prove that we can consider as noisy input a fractional Brownian motion with Hurst parameter bigger than $1/2$. For this process, we prove that the required assumptions {\bf(S1)}, {\bf (S2)} and {\bf(S3)} are fulfilled.

Given $H\in (0,1)$, a continuous centered Gau{\ss}ian process
$\beta^H=(\beta^H(t), t\in\mathbb{R})$, with  covariance function
\begin{equation*}
    R(s,t)=\frac{1}{2}(|t|^{2H}+|s|^{2H}-|t-s|^{2H}), \qquad s,\,t \in \RR,
\end{equation*}
on an appropriate probability space, is called a two--sided one-dimensional fractional Brownian
motion (fBm) with Hurst parameter $H$. When $H=1/2$, $B$ is the standard Brownian motion.

Assume that $Q$ is a bounded and symmetric linear operator on $V$ and that $Q$ is of trace class, i.e., for a complete orthonormal basis $(e_i)_{i\in {\mathbb N}}$ in $V$ there exists a sequence of nonnegative numbers $(q_i)_{i\in {\mathbb N}}$ such that $\text{tr}(Q):=\sum_{i=1}^{\infty}q_i <\infty$. Then a continuous $V$-valued fractional Brownian motion $B^H$ with  covariance operator $Q$ and Hurst parameter $H$ is defined by
\begin{equation*}
   B^H(t)=\sum_{i=1}^{\infty} \sqrt{q_i} \beta_i^H(t) e_i,\quad t\in\mathbb{R},
\end{equation*}
where $(\beta_i^H)_{i\in{\NN}}$ is a sequence of stochastically independent one-dimensional fBms with the same Hurst parameter $H$.

We consider a canonical version of this process given by the probability space $(C_0(\RR;V),\bB(C_0(\RR;V)),\PP)$ where $\PP$ is the Gau{\ss}-measure generated by $B^H$.
On this probability space we can also introduce the  shift
$$\theta_t\omega(\cdot)=\omega(\cdot+t)-\omega(t).$$
It is known that $(C_0(\RR;V),\bB(C_0(\RR;V)),\PP,\theta)$ is an ergodic metric dynamical system, see \cite{MasSchm04} and \cite{GS11}. This canonical process has a version which is $\beta^{\prime\prime}$-H\"older continuous on any compact
interval $[-k,k]$ for $ \beta^{\prime\prime}< H$.\\

For $H>1/2$, in what follows we set the parameters $1/2<\beta<\beta^\prime<\beta^{\prime\prime}<H$. Let $\Omega$ be the $(\theta_t)_{t\in \RR}$-invariant set of paths $\omega:\RR\to V$ which are $\beta^{\prime\prime}$-H{\"o}lder continuous on any compact subinterval of $\RR$ and are zero at zero. Then, see \cite{CGGSch14}, $\Omega \in\bB(C_0(\RR;V))$ and it is $(\theta_t)_{t\in \RR}$-invariant.

\begin{lemma}\label{l2} Let $\omega \in \Omega$.
The mapping $t\mapsto \ltn\omega\rtn_{\beta^\prime,0,t}$ is continuous on $\RR^+$,  and the mapping
$t\mapsto \ltn\omega\rtn_{\beta^\prime,t,0}$ is continuous on $\RR^-$.
\end{lemma}
\begin{proof}
We only prove the continuity of the first mapping, the second mapping can be treated analogously. If $t\geq t_0$, define $\omega^{t_0}$ given by
\begin{align}\label{ome}
\begin{split}
\omega^{t_0}(s)=\left\{
\begin{array}{lcl}
\omega(s)&:&{\rm for }\ s< t_0,\\
\omega(t_0)&:&{\rm for }\ s\geq t_0.
\end{array}\right.
\end{split}
\end{align}
Thus, for $t \geq t_0$,
\[ \bigg| \ltn \omega\rtn_{\beta^\prime,0,t}-\ltn \omega\rtn_{\beta^\prime,0,t_0}\bigg|
 =\bigg|\ltn \omega\rtn_{\beta^\prime,0,t}-\ltn \omega^{t_0}\rtn _{\beta^\prime,0,t} \bigg| \le \ltn \omega- \omega^{t_0}\rtn _{\beta^\prime,0,t}=\ltn \omega\rtn _{\beta^\prime,t_0,t},\]
hence
\begin{align*}
\limsup_{t\downarrow t_0}\bigg| \ltn \omega\rtn_{\beta^\prime,0,t}-\ltn \omega\rtn_{\beta^\prime,0,t_0}\bigg|& \leq     \limsup_{t\downarrow t_0} \left( \ltn \omega\rtn_{\beta^{\prime \prime},t_0,t} (t-t_0)^{\beta^{\prime \prime}-\beta^{\prime }} \right) \\ &\leq     \limsup_{t\downarrow t_0} \ltn \omega\rtn_{\beta^{\prime \prime},0,t}  \lim_{t\downarrow t_0} (t-t_0)^{\beta^{\prime \prime}-\beta^{\prime }} =0,
\end{align*}
and, by the same argument, $\liminf_{t\downarrow t_0}\big| \ltn \omega\rtn_{\beta^\prime,0,t}-\ltn \omega\rtn_{\beta^\prime,0,t_0}\big|=0$, therefore
\[\lim_{t\downarrow t_0} \ltn \omega\rtn_{\beta^\prime,0,t}=\ltn \omega\rtn_{\beta^\prime,0,t_0}.\]
The case $t\leq t_0$ is similar.
\end{proof}
Now for a given constant $\mu>0$,  define
\begin{eqnarray}
T(\omega)&=&\inf\{\tau>0:c_{\alpha,\beta,\beta^\prime}  c_{DG}\ltn\omega\rtn_{\beta^\prime,0,\tau}\tau^{\beta^\prime}+ c_{DF}\tau> c_{DF}\mu\},\nonumber \\[-1.5ex]
\label{eq29}\\[-1.5ex]
\widehat T(\omega)&=&\sup\{\tau<0:c_{\alpha,\beta,\beta^\prime}  c_{DG}\ltn\omega\rtn_{\beta^\prime,\tau,0}(-\tau)^{\beta^\prime}- c_{DF}\tau> c_{DF}\mu\}.\nonumber
\end{eqnarray}

\begin{lemma}\label{l3} The mappings $T,  \, \widehat T: \Omega \rightarrow \mathbb{R}$ defined by \eqref{eq29} are measurable.
 Moreover,  we have $T(\omega), \, -\widehat T(\omega)\in (0,\mu]$ for all $\omega \in \Omega$ and the following relations are fulfilled:
\begin{equation*}
T(\omega)=-\widehat T(\theta_{T(\omega)}\omega),\quad \widehat T(\omega)=- T(\theta_{\widehat T(\omega)}\omega), \qquad \omega \in \Omega.
\end{equation*}
\end{lemma}

\begin{proof}
Recall that $\Omega$ is the set of paths $\omega:\RR\to V$ which are $\beta^{\prime\prime}$-H{\"o}lder continuous on any compact subinterval of $\RR$ and are zero at zero.
For $(\omega,t)\in C^{\beta^{\prime\prime}}([0,\mu];V)\times [0,\mu]$,  define the function
$$f(\omega,t)=c_{\alpha,\beta,\beta^\prime}  c_{DG}\ltn\omega\rtn_{\beta^\prime,0,t}t^{\beta^\prime}+ c_{DF}t- c_{DF}\mu$$
and similarly, for $(\omega,t)\in C^{\beta^{\prime\prime}}([-\mu,0];V)\times [-\mu,0]$, we can define the mapping
$$g(\omega,t)=c_{\alpha,\beta,\beta^\prime}  c_{DG}\ltn\omega\rtn_{\beta^\prime,t,0}(-t)^{\beta^\prime}- c_{DF}t- c_{DF}\mu.$$

In virtue of Lemma \ref{l2}, $f(\omega, \cdot)$ is continuous, and trivially,  it is strictly increasing in $t$, such that $f(\omega,0)<0$ and $f(\omega,\mu) \geq 0$ (indeed $f(\omega,\mu) = 0$ when we do not have noise, namely, when $\omega\equiv 0$. Otherwise, $f(\omega,\mu) > 0$). Hence, for each $\omega$ there exists a unique $T(\omega) \in (0, \mu]$ such that $f(\omega,T(\omega)) = 0.$ Thus, we can reformulate $T(\omega)$ as
 \begin{equation}\label{stnew}
c_{\alpha,\beta,\beta^\prime} c_{DG} \ltn\omega\rtn_{\beta^\prime,0,T(\omega)}T(\omega)^{\beta^\prime}+ c_{DF}T(\omega)=  c_{DF}\mu.
\end{equation}

Moreover, $T(\omega)$ is measurable which follows from the fact that $f(\omega,\cdot)$ is strictly increasing w.r.t. $t$ and
\begin{align*}
  \{\omega \in \Omega: T(\omega)< \kappa\}&=\{\omega\in \Omega:c_{\alpha,\beta,\beta^\prime}  c_{DG}\ltn\omega\rtn_{\beta^\prime,0,\kappa} \kappa^{\beta^\prime}+ c_{DF} \kappa>  c_{DF}\mu\}\\
  &=\ltn\cdot\rtn_{\beta^\prime,0,\kappa}^{-1}\bigg(\frac{c_{DF}(\mu- \kappa)}{c_{\alpha,\beta,\beta^\prime}  c_{DG}  \kappa^{\beta^\prime}},\infty\bigg),
  \end{align*}
and the right hand side of this expression is measurable due to the measurability of $\ltn\cdot\rtn_{\beta^\prime,0,\kappa}$ with respect to $\bB(C_0(\RR;V))$.

\smallskip
Using the same arguments for $g$, we can show that $\widehat T(\omega)$ defined by (\ref{eq29}) belongs to $[-\mu,0)$ and it is measurable. In particular we obtain
 \begin{equation*}
c_{\alpha,\beta,\beta^\prime}  c_{DG}\ltn\theta_{T(\omega)}\omega\rtn_{\beta^\prime, \widehat T(\theta_{T(\omega)}\omega),0}(- \widehat T(\theta_{T(\omega)}\omega))^{\beta^\prime}- c_{DF} \widehat T(\theta_{T(\omega)}\omega)=   c_{DF}\mu,
\end{equation*}
hence
$$   c_{\alpha,\beta,\beta^\prime}  c_{DG} \ltn \omega\rtn_{\beta^\prime, T(\omega) + \widehat{T}(\theta_{T(\omega)}\omega)   , T(\omega)}(-\widehat T(\theta_{T(\omega)}\omega))^{\beta^\prime} - c_{DF} \widehat T(\theta_{T(\omega)}\omega) =  c_{DF} \mu ,  $$
and, comparing this last equality with (\ref{stnew}), we derive that
 $T(\omega)=-\widehat T(\theta_{T(\omega)}\omega)$ due to the uniqueness of solution. The last statement of the lemma follows analogously.
\end{proof}

From the definition of $T(\omega)$ and $\widehat T(\omega)$, identifying $T(\omega)$ with $T_1(\omega)$ and $\widehat T(\omega)$ with $T_{-1}(\omega)$, we can define a sequence of  stopping times $(T_i(\omega))_{i\in \ZZ}$ as follows
\begin{eqnarray}\label{stop}
    T_i(\omega)=\left\{\begin{array}{lcr}0&:&i=0,\\T_{i-1}(\omega)+T(\theta_{T_{i-1}(\omega)}\omega)&:&i\in \NN,\\
    T_{i+1}(\omega)+\widehat T(\theta_{T_{i+1}(\omega)}\omega)&:&i\in -\NN.\end{array}\right.
\end{eqnarray}
The stopping times satisfy {\bf (S1)} by construction: since $T_{i+1}(\omega)-T_i(\omega)=T_1(\theta_{T_i(\omega)}\omega)$, in virtue of (\ref{stnew}) we have, for  $i \in \mathbb{N}$,
$$ 	c_{\alpha,\beta,\beta^\prime}  c_{DG}\ltn\theta_{T_i(\omega)}\omega\rtn_{\beta^\prime,0,T_{i+1}(\omega)-T_i(\omega)}(T_{i+1}(\omega)-T_i(\omega))^{\beta^\prime} +  c_{DF}(T_{i+1}(\omega)-T_i(\omega)) =  c_{DF}\mu.$$
Analogously,  for $ i \in -\NN$,
\begin{equation*}
c_{\alpha,\beta,\beta^\prime}  c_{DG}\ltn\theta_{T_i(\omega)}\omega\rtn_{\beta^\prime,  T_{i-1}(\omega)-T_{i}(\omega),0}( T_{i}(\omega)-T_{i-1}(\omega))^{\beta^\prime}- c_{DF} (T_{i-1}(\omega)-T_{i}(\omega))= c_{DF} \mu.\end{equation*}

\begin{lemma}\label{l1}
The sequence $(T_i(\omega))_{i\in \ZZ}$ satisfies the following properties:
\begin{enumerate}
\item[(i)] For $i,\,j\in \ZZ$ we have
\[
T_i(\omega)+T_j(\theta_{T_i(\omega)}\omega)=T_{i+j}(\omega).
\]
\item[(ii)]  Let $t_1\le t_2$ for $t_1,\,t_2\in\RR$. Then we have
\begin{align}\label{ek}
\begin{split}
    t_1+&  T(\theta_{t_1}\omega)\le t_2+ T(\theta_{t_2}\omega),\\
    t_1+& \widehat T(\theta_{t_1}\omega)\le t_2+ \widehat T(\theta_{t_2}\omega).
  \end{split}
\end{align}
If $t_1< t_2$, then the corresponding inequalities are strict, that is,
\begin{align}\label{ekb}
\begin{split}
    t_1+&  T(\theta_{t_1}\omega)< t_2+ T(\theta_{t_2}\omega),\\
    t_1+& \widehat T(\theta_{t_1}\omega)< t_2+ \widehat T(\theta_{t_2}\omega).
  \end{split}
\end{align}
\item[(iii)] If $t_1 \leq t_2< t_1+T(\theta_{t_1}\omega)$, then it holds
\begin{align}\label{eq33}
\begin{split}
t_1 \le t_2& < t_1+T(\theta_{t_1}\omega) \le t_2+
T(\theta_{t_2}\omega)\\
& < t_1+ T_2(\theta_{t_1}\omega)\leq  t_2+ T_2(\theta_{t_2}\omega)\\
& < \qquad \cdots  \qquad \leq \qquad \cdots
    \end{split}
    \end{align}
\end{enumerate}
\end{lemma}

Let us emphasize that (\ref{eq33}) can be presented in several ways, depending on the relationship between $t_1$, $t_2$ and $ t_1+T(\theta_{t_1}\omega)$, and accordingly we use the sign $<$ or $\leq$.

\begin{proof}
Thanks to the definition of the stopping times (\ref{stop}), by induction the first of the properties follows easily.

Let us prove (ii) by contradiction, recalling first that the constant $c_{DF} >0$. Assume that $t_1\leq t_2$ and $t_1+T(\theta_{t_1}\omega)>t_2 + T(\theta_{t_2}\omega)$. This means that
\begin{eqnarray*}
 c_{DF}  \mu&=&c_{\alpha,\beta,\beta^\prime}  c_{DG}\ltn \theta_{t_2}\omega\rtn_{\beta^\prime,0,T(\theta_{t_2}\omega)}T(\theta_{t_2}\omega)^{\beta^\prime}+ c_{DF} T(\theta_{t_2}\omega)\\
   & =&c_{\alpha,\beta,\beta^\prime}  c_{DG}\sup_{t_2\le s < t\le t_2+T(\theta_{t_2}\omega)}\frac{\|\omega(t)-\omega(s)\|}{(t-s)^{\beta^\prime}} T(\theta_{t_2}\omega)^{\beta^\prime}+ c_{DF} T(\theta_{t_2}\omega)\\
       &<&c_{\alpha,\beta,\beta^\prime}  c_{DG}\sup_{t_1\le s < t\le t_1+T(\theta_{t_1}\omega)}\frac{\|\omega(t)-\omega(s)\|}{(t-s)^{\beta^\prime}} T(\theta_{t_1}\omega)^{\beta^\prime}+ c_{DF} T(\theta_{t_1}\omega)\\
    &=&c_{\alpha,\beta,\beta^\prime}  c_{DG}\ltn \theta_{t_1}\omega\rtn_{\beta^\prime,0,T(\theta_{t_1}\omega)}T(\theta_{t_1}\omega)^{\beta^\prime}+  c_{DF}T(\theta_{t_1}\omega)= c_{DF} \mu,
\end{eqnarray*}
which is a contradiction, and therefore $t_1 + T(\theta_{t_1}\omega)\leq t_2 + T(\theta_{t_2}\omega)$. The second part of (ii) can be proved analogously.

Finally, iterating (\ref{ek}) and (\ref{ekb}) the chain of inequalities (\ref{eq33}) follows.
\end{proof}

\medskip

In order to check the condition {\bf(S2)} for the fractional Brownian motion, in what follows we would like to find a lower estimate of the number of stopping times that lie in any interval of length $t$. To do that, we first consider the number of stopping times $N(\omega)$ in the interval $(0,\mu]$ and establish an upper bound for it. Next, we will check that in any interval of length $\mu$ the number of stopping times coincides either with the number of stopping times in the interval $(0,\mu]$ or this number plus one. However, it will turn out that by choosing a fractional Brownian motion such that the corresponding trace of its covariance operator is sufficiently small, then on a interval of the type $(0,\mu m]$ the number of stopping times coincides with $m$ or $m+1$, which will be the key property to check {\bf(S2)}.
\medskip

\begin{lemma}\label{l1n}
Given $\omega\in \Omega$, let $N(\omega)$ denote the number of stopping times in $(0,\mu]$, i.e.

$$ N(\omega)= \max \{i \in \NN: \,T_{i}(\omega) \leq \mu \}.$$
Then we have
\begin{equation*}
 N(\omega)\leq \mu K(\omega, \mu),
\end{equation*}
with
\begin{align}\label{K}
K(\omega, \mu)=  \bigg( \frac{c_{\alpha,\beta,\beta^\prime} c_{DG} \ltn\omega\rtn_{\beta^{\prime\prime},0,\mu}+ c_{DF}\mu^{1-\beta^{\prime\prime}}}{c_{DF}\mu}\bigg)^{\frac{1}{\beta^{\prime \prime}}}.
\end{align}
Moreover,
\begin{align}\label{EK}
\limsup_{n \to\infty}\frac{1}{\mu n}\sum_{j=0}^{n-1} N(\theta_{j\mu}\omega)\le\EE\bigg(\sup_{r\in [-\mu,0]}K(\theta_r \omega, \mu)\bigg)=:\frac{1}{\mu d}
\end{align}
for all $\omega$ in  a $(\theta_t)_{t\in\RR}$ invariant set of full measure
where $d=d(\mu,{\rm tr}(Q)) \in (0,1]$.
\end{lemma}
\begin{proof}
As a direct consequence of the definition of the stopping time $T(\omega)$ we derive
\begin{align*}
    c_{DF} \mu&= c_{\alpha,\beta,\beta^\prime}  c_{DG}\sup_{0\le s<t\le T(\omega)}\frac{\|\omega(t)-\omega(s)\|}{(t-s)^{\beta^\prime}}T(\omega)^{\beta^\prime}+ c_{DF}T(\omega)\\
    &\le T(\omega)^{\beta^\prime} \bigg(c_{\alpha,\beta,\beta^\prime}  c_{DG}\sup_{0\le s<t\le T(\omega)}\frac{\|\omega(t)-\omega(s)\|}{(t-s)^{\beta^{\prime\prime}}}T(\omega)^{\beta^{\prime \prime}-\beta^\prime}+ c_{DF}T(\omega)^{1-\beta^\prime}\bigg)\\
    &= T(\omega)^{\beta^\prime} T(\omega)^{\beta^{\prime\prime}-\beta^\prime} \bigg(c_{\alpha,\beta,\beta^\prime}  c_{DG}\ltn \omega \rtn_{\beta^{\prime\prime}, 0,T(\omega)} + c_{DF}T(\omega)^{1-\beta^{\prime \prime}}\bigg)\\
    &\le T(\omega)^{\beta^{\prime\prime}} \bigg(c_{\alpha,\beta,\beta^\prime}  c_{DG}\ltn \omega \rtn_{\beta^{\prime\prime}, 0,\mu} +c_{DF} \mu^{1-\beta^{\prime \prime}}\bigg)
\end{align*}
and therefore,
\begin{equation*}
  T(\omega)\ge \bigg(\frac{c_{DF}\mu}{c_{\alpha,\beta,\beta^\prime}  c_{DG}\ltn \omega \rtn_{\beta^{\prime\prime}, 0,\mu} + c_{DF}\mu^{1-\beta^{\prime \prime}}}  \bigg)^{\frac{1}{\beta^{\prime \prime}}}.
\end{equation*}

Furthermore, the same inequality holds true if we replace $T(\omega)$ by $T(\theta_{T_{j}(\omega)} \omega)$, provided that $T(\theta_{T_{j}(\omega)} \omega)+T_{j}(\omega)=T_{j+1}(\omega)\leq \mu$, since then
\begin{align*}
  c_{DF}  \mu&=c_{\alpha,\beta,\beta^\prime}  c_{DG}\sup_{0\le s<t\le T(\theta_{T_{j}(\omega)}\omega)}\frac{\|\theta_{T_{j}(\omega)}\omega(t)-\theta_{ T_{j}(\omega)}\omega(s)\|}{(t-s)^{\beta^\prime}}T(\theta_{T_{j}(\omega)}\omega)^{\beta^\prime}+c_{DF}T(\theta_{T_{j}(\omega)}\omega)\\
    &\le c_{\alpha,\beta,\beta^\prime}  c_{DG}\sup_{T_{j}(\omega)\le s<t\le T_{j+1}(\omega)}\frac{\|\omega(t)-\omega(s)\|}{(t-s)^{\beta^{\prime\prime}}}T(\theta_{T_{j}(\omega)}\omega)^{\beta^{\prime\prime}}+c_{DF}T(\theta_{T_{j}(\omega)}\omega)\\
    &= T(\theta_{T_{j}(\omega)}\omega)^{\beta^{\prime\prime}}
(c_{\alpha,\beta,\beta^\prime}  c_{DG}\ltn\omega\rtn_{\beta^{\prime\prime},T_{j}(\omega),T_{j+1}(\omega)}+ c_{DF}T(\theta_{T_{j}(\omega)}\omega)^{1-\beta^{\prime \prime}})\\
&\le T(\theta_{T_{j}(\omega)}\omega)^{\beta^{\prime\prime}}
(c_{\alpha,\beta,\beta^\prime}  c_{DG}\ltn\omega\rtn_{\beta^{\prime\prime},0,\mu}+ c_{DF}\mu^{1-\beta^{\prime \prime}}),
\end{align*}
and hence we also have
\begin{equation*}
  T(\theta_{T_{j}(\omega)}\omega)\ge \bigg(\frac{c_{DF} \mu}{c_{\alpha,\beta,\beta^\prime} c_{DG}\ltn \omega \rtn_{\beta^{\prime\prime}, 0,\mu} + c_{DF}\mu^{1-\beta^{\prime \prime}}}  \bigg)^{\frac{1}{\beta^{\prime \prime}}},\; \text{ whenever  }\; T_{j+1}(\omega)\le \mu.
\end{equation*}
By definition, $N(\omega)$ is the largest natural number such that $T_{N(\omega)}(\omega)\leq \mu$ . Using Lemma \ref{l1}(a) we have
\begin{align*}
   \mu&\geq T_{N(\omega)}(\omega) =\sum_{j=0}^{N(\omega)-1}T(\theta_{T_{j}(\omega)}(\omega)\ge  N(\omega) \bigg(\frac{ c_{DF}\mu}{c_{\alpha,\beta,\beta^\prime} c_{DG}\ltn \omega \rtn_{\beta^{\prime\prime},0, \mu} + c_{DF}\mu^{1-\beta^{\prime \prime}}}  \bigg)^{\frac{1}{\beta^{\prime \prime}}}
\end{align*}
and consequently the estimate of $N$ follows.

On the other hand, note  that $\sup_{r\in[-\mu,0]}\ltn \theta_r\omega\rtn_{\beta^{\prime\prime},0,\mu}\le \ltn\omega\rtn_{\beta^{\prime\prime},-\mu,\mu}$ and by Kunita \cite{Kun90}, Theorem 1.4.1, the right hand side  of this inequality has finite moments of any order, hence $ \EE\bigg(\sup_{r\in [-\mu,0]}K(\theta_r \omega, \mu)\bigg)<\infty$. Indeed,
\[
N(\theta_{j\mu}\omega) \leq \mu K(\theta_{j\mu} \omega, \mu) \le \int_0^\mu \sup_{r\in[-\mu,0]} K( \theta_{r+q+j\mu}\omega,\mu ) dq ,
\]
and this together with (\ref{change}) imply
\begin{align*}
     \sum_{j=0}^{n-1} N(\theta_{j\mu}\omega)&\leq \sum_{j=0}^{n-1}\int_0^\mu \sup_{r\in[-\mu,0]} K( \theta_{r+q+j\mu}\omega,\mu ) dq=\int_0^{\mu n} \sup_{r\in[-\mu,0]} K( \theta_{r+q}\omega,\mu ) dq,
           \end{align*}
hence
\begin{align*}
\limsup_{n \to\infty}&\frac{1}{\mu n}\sum_{j=0}^{n-1} N(\theta_{j\mu}\omega)\le\lim_{n\to\infty}\frac{1}{\mu n}\int_0^{\mu n} \sup_{r\in[-\mu,0]} K( \theta_{r+q}\omega,\mu ) dq=\EE \bigg(\sup_{r\in [-\mu,0] }K(\theta_r \omega, \mu )\bigg).
 \end{align*}
By the ergodic theorem for continuous time, there is a $(\theta_t)_{t\in\RR}$ invariant set of full measure where this convergence takes place.

In particular, we deduce that
\begin{align*}
\frac{1}{d}=\mu \EE\bigg(\sup_{r\in [-\mu,0]} K(\theta_r \omega, \mu)\bigg)\geq \mu \EE \bigg( \frac{\mu^{1-\beta^{\prime\prime}}}{\mu}\bigg)^{\frac{1}{\beta^{\prime \prime}}}=1.
\end{align*}
\end{proof}

In the following result we analyze the number of stopping times in any interval of length $\mu$.

\begin{lemma}\label{l2n}
Let $\omega \in \Omega$, $j\in \NN$ and let $M_j(\omega)$ be the number of stopping times in the interval $(j\mu,(j+1)\mu]$. Then
$$ M_j(\omega)=N(\theta_{j \mu }\omega)+1_{\{M_j(\cdot)>N(\theta_{j\mu}\cdot)\}}(\omega)$$
i.e.
$$
M_j(\omega)\in \{   N(\theta_{j\mu}\omega),  N(\theta_{j\mu}\omega) +1 \}, $$
where $N$ has been introduced in Lemma \ref{l1n}.
\end{lemma}
\begin{proof}
Consider $j\in \NN$ and $T_k(\omega) \in(j\mu,(j+1)\mu]$. Assume that $k_j(\omega)$ is the smallest natural number such that  $T_{k_j}(\omega)> j\mu$, hence
$$T_{k_j(\omega)-1}(\omega) \leq j\mu <T_{k_j(\omega)}(\omega).$$
Note that by definition of $k_j$ we have
$$k_j(\omega)=\sum_{i=0}^{j-1}M_i(\omega)+1.$$
Applying \eqref{eq33} with $t_1=T_{k_j(\omega)-1}(\omega)$ and $t_2=j\mu$, so that
\begin{align*}
& t_1+T(\theta_{t_1} \omega)=T_{k_j(\omega)-1}(\omega)+ {T}(\theta_{T_{k_j(\omega)-1}}\omega)=T_{k_j(\omega)}(\omega), \\
& t_2 + T(\theta_{t_2}\omega)=j \mu + T(\theta_{j \mu }\omega) = j \mu +  T_{1}(\theta_{j \mu }\omega), \\
& t_1+T(\theta_{t_1} \omega) +T (\theta_{t_1 +  T(\theta_{t_1} \omega)}\omega)= T_{k_j(\omega)+1}(\omega), \\
& t_2 +  T(\theta_{t_2}\omega)  + T (\theta_{t_2 +  T(\theta_{t_2} \omega)}\omega)   =j \mu + T_{2}(\theta_{j \mu }\omega),
\end{align*}
we obtain
\begin{eqnarray} \label{chain_stop}
 T_{k_j(\omega)-1}(\omega)  \leq  j \mu  <T_{k_j(\omega)}(\omega) \leq j\mu+T_{1}(\theta_{j\mu}\omega)  <  T_{k_j(\omega)+1}(\omega) \leq   j\mu+T_{2}(\theta_{j\mu}\omega).
\end{eqnarray}


Iterating  \eqref{chain_stop} gives
\begin{align} \label{schachtel}     T_{k_j(\omega)+\ell-1}(\omega) \leq j\mu+T_{\ell}(\theta_{j\mu}\omega)  <  T_{k_j(\omega)+\ell}(\omega), \qquad \ell =0,1,2 \ldots . \end{align}
Recall that $T_{k_j(\omega)}(\omega)$ is the first stopping time in $(j\mu, (j+1) \mu]$, and, by definition of $M_j(\omega)$, the stopping time
$T_{k_j(\omega)+M_j(\omega)-1}(\omega)$ is the last one, i.e.
\begin{align} T_{k_{j+1}(\omega)}(\omega) = T_{k_j(\omega)+M_j(\omega)}  (\omega) > (j+1) \mu \geq T_{
{k_{j+1}(\omega)}-1}(\omega)= T_{k_j(\omega)+M_j(\omega)-1}(\omega). \label{schachtel_2} \end{align}

So it follows from \eqref{schachtel} that
\begin{align*}     j \mu +  T_{M_{j}(\omega)-1}(\theta_{j\mu}\omega) <   T_{k_j(\omega)+M_j(\omega)-1}  (\omega) \leq (j+1) \mu .\end{align*}
Therefore we have
\begin{align*}      T_{M_{j}(\omega)-1}(\theta_{j\mu}\omega)  <  \mu \end{align*}
and so
$$  N(\theta_{j\mu}\omega) \geq M_j(\omega)-1. $$
On the other hand, we obtain also from  \eqref{schachtel} and \eqref{schachtel_2} that
$$  (j+1) \mu <  T_{k_j(\omega)+M_j(\omega)}(\omega) \leq j\mu+T_{M_{j}(\omega)+1}(\theta_{j\mu}\omega),$$ which implies
$$ T_{M_{j}(\omega)+1}(\theta_{j\mu}\omega) >\mu $$
and so
$$  N(\theta_{j\mu}\omega) \leq M_j(\omega). $$

\end{proof}

\smallskip

\begin{lemma}\label{le1}
Let $\mu$ given as in {\bf(S1)}. For any small $\delta>0$ such that ${\rm tr}(Q)<\delta$ there exist a $(\theta_t)_{t\in \RR}$-invariant set $\Omega^\prime\in \fF$ of full measure and $\bar d=\bar d(\mu,{\rm tr}(Q))$, such that for $\omega\in\Omega^\prime$,
\begin{equation*}
\limsup_{n\to\infty}\frac{1}{n}\sum_{j=0}^{n-1} 1_{\{M_j(\cdot)>N(\theta_{j\mu}\cdot)\}}(\omega)\le \bar d.
\end{equation*}
In particular, for fixed $\mu$, for any $\varepsilon >0$ there exists $\delta>0$ with ${\rm tr}(Q)<\delta$, such that for $\omega \in \Omega^\prime$
\begin{equation*}
\limsup_{n\to\infty}\frac{1}{n}\sum_{j=0}^{n-1} 1_{\{M_j(\cdot)>N(\theta_{j\mu}\cdot)\}}(\omega)\le \varepsilon.
\end{equation*}
\end{lemma}
\begin{proof}
Let $\varepsilon>0$ be given. Choose $m>2,\, m\in\NN$ large enough such that
\begin{equation*}
  \frac{1}{m}<\frac{\varepsilon}{2},
\end{equation*}
and define
\begin{align}\label{eq9}
\begin{split}
A&=\{\omega\in\Omega:  \sup_{\rho\in [-\mu m,\mu(m-1)]}\ltn\omega\rtn_{\beta^\prime,\rho,\rho+\mu}<c\}\in\fF,\\
B&=\{\omega\in\Omega:  \sup_{\rho\in [0,\mu(m-1)]}\ltn\omega\rtn_{\beta^\prime,\rho,\rho+\mu}<c\}\in\fF.
\end{split}
\end{align}
We choose ${\rm tr}(Q)$ small enough such that the set of trajectories $A$ fulfills $\PP(A)>1-\varepsilon/2$. In $A$ and $B$ the constant $c$ is chosen such that the solution $\tau$ of the following equation related to \eqref{stnew}
\begin{equation*}
c_{\alpha,\beta,\beta^\prime} c_{DG} c\tau^{\beta^\prime}+c_{DF}\tau=c_{DF}\mu
\end{equation*}
is larger than $\mu-\frac{\mu}{2(m+1)}$.

\smallskip
For $i\in\ZZ^+$ define
\begin{align*}
  s_1(i,\omega)&=\sum_{i^\prime=0}^{m-1} 1_{\{M_{i m+i^\prime}(\cdot)>N(\theta_{(i m+i^\prime)\mu}\cdot)\}}(\omega),\\
  s_2(\omega)&=\sum_{i^\prime =0}^{m-1} 1_{\{M_{i^\prime} (\cdot)>N(\theta_{i^\prime \mu}\cdot)\}}(\omega).
\end{align*}
The sequence $s_1(i,\omega)$ is not stationary, thus we are going to use the sequence given by $s_2(\theta_i\omega)$ to control the nonstationarity of $s_1$.

For $i\in\ZZ^+$, suppose that $\theta_{i m\mu}\omega\in A$ then $\theta_{i m\mu}\omega\in B$. Let
\begin{equation}\label{eq25}
  \{T_k(\omega), k\in\{\check k,\cdots,\hat k\}\}
\end{equation}
be the nonempty set of all stopping times in $(i m\mu, i(m+1)\mu]$. Note that
\begin{equation*}
  \ltn\omega\rtn_{\beta^\prime,\rho+i m \mu,\rho+ i(m+1)\mu}=\ltn\theta_{im\mu}\omega\rtn_{\beta^\prime,\rho,\rho+\mu}<c,\quad
  \rho\in [0,\mu(m-1)].
\end{equation*}
We claim that the number of stopping times in $(i m\mu, i(m+1)\mu]$ is either $m$ or $m+1$. Assume we have $m+2$ or more elements in \eqref{eq25}. Then the sum of the $m+1$ distances between
neighboring stopping times exceeds $m\mu$ which is a contradiction because we only consider stopping times in $(i m\mu, i(m+1)\mu]$.
We have that
\begin{equation*}
  M_{im+i^\prime}(\omega)\ge N(\theta_{(im+i^\prime)\mu}\omega)=1\quad\text{for }i^\prime=0,\cdots,m-1.
\end{equation*}
Indeed the number of stopping times forming the number $N(\theta_{(im+i^\prime)\mu}\omega)$ is between $\mu-\mu/(2(m+1)$ and $\mu$ such that there is exactly one in those intervals $((im+i^\prime)\mu,(im+i^\prime+1)\mu]$. According to Lemma \ref{l2n} we obtain that $M_{im+i^\prime}(\omega)\in\{1,2\}$ and there can only be one of those terms which is 2, therefore
\begin{align*}
  s_1(i,\omega)\in \{0,1\}.
\end{align*}
Let
\begin{equation*}
  \{T_{i^\prime}(\theta_{i m \mu}\omega), i^\prime\in\{1,\cdots, \hat i\}\}
\end{equation*}
be the nonempty set of all stopping times  in $(0, m\mu]$.
By the same arguments
\begin{equation*}
  s_2(\theta_{im\mu}\omega)\in \{0,1\},
\end{equation*}
which means that $s_1(\omega),\,s_2(\theta_{im\mu}\omega)$ do not have  more than one term equal to one.
Then by the comparison principle of Lemma \ref{l1}
\begin{align*}
im\mu< T_{\check k}(\omega)<& im\mu +T_1(\theta_{i m\mu}\omega)< T_{\check k+1}(\omega)< i m\mu+ T_2(\theta_{i m\mu}\omega)< \cdots.
\end{align*}
Hence the number of stopping times $T_{i^\prime}(\theta_{im\mu}\omega)$ in $(0, m\mu]$ is smaller than or equal to the number of stopping times $T_k(\omega)\in
(i m\mu, i(m+1)\mu]$ which is in $\{m,m+1\}$. Thus $s_1(i,\omega)\ge s_2(\theta_{im\mu}\omega)$.

Define
\begin{equation*}
  Y^m(\omega)=
  \left\{\begin{array}{lll}1&:& \omega\in B^c,\\
  1/m&:&\omega\in B.
  \end{array}\right.
\end{equation*}

In particular we have
\begin{equation*}
 \frac{1}{m} s_2(\theta_{im\mu}\omega)\le Y^m(\theta_{i m\mu}\omega)\quad \text{for }\omega\in\Omega.
\end{equation*}
If $s_1(i,\omega)\not=s_2(\theta_{im\mu}\omega)$ for $\theta_{i\mu}\omega\in B$ then $s_2(\theta_{im\mu}\omega)=0$ and $s_1(i,\omega)=1$ thus
\begin{equation*}
  \frac{1}{m}\sum_{i^\prime=0}^{m-1} 1_{\{M_{i m+i^\prime}(\cdot)>N(\theta_{(i m+i^\prime)\mu}\cdot)\}}(\omega)=\frac{1}{m}s_1(i,\omega)\le Y^m(\theta_{i m\mu}\omega)\quad \text{for }\omega\in\Omega.
\end{equation*}

Without loss of generality we can assume that $n=n^\prime m$ for $n^\prime\in\NN$. Then

\begin{align*}
\limsup_{n\to\infty} & \frac{1}{n}\sum_{j=0}^{n-1} 1_{\{M_j(\cdot)>N(\theta_{j\mu}\cdot)\}}(\omega)\\
=\limsup_{n^\prime\to\infty} & \frac{1}{n^\prime}\sum_{i=0}^{n^\prime-1}\frac{1}{m}\sum_{i^\prime=0}^{m-1} 1_{\{M_{im+i^\prime}(\cdot)>N(\theta_{(im+i^\prime)\mu}\cdot)\}}(\omega)\\
\le \limsup_{n^\prime\to\infty} & \frac{1}{n^\prime}\sum_{i=0}^{n^\prime-1}Y^m(\theta_{i m \mu}\omega).
\end{align*}

However
\begin{equation*}
  Y^m(\omega)\le \frac{1}{\mu m}\int^{\mu m}_0\sup_{r\in[-\mu m,0]}Y^m(\theta_{r+q}\omega)dq,
\end{equation*}
such that
\begin{equation*}
\frac{1}{n^\prime}\sum_{i=0}^{n^\prime-1}Y^m(\theta_{i m \mu}\omega)\leq\frac{1}{n^\prime\mu m}\int^{n^\prime\mu m}_0\sup_{r\in[-\mu m,0]}Y^m(\theta_{r+q}\omega)dq.
\end{equation*}
Finally,
\begin{align}\label{nl}
\begin{split}
 \lim_{t\to\pm\infty} &\frac{1}{t}\int_0^t\sup_{r\in[-\mu m,0]}Y^m(\theta_{r+q}\omega)dq=\EE\sup_{r\in[-\mu m,0]}Y^m(\theta_{r}\omega)=:\bar d,
 \end{split}
\end{align}
where this convergence holds true on a $(\theta_t )_{t\in \RR}$ invariant set of full measure. $\bar d$ depends on $\mu,\,m$ and ${\rm tr}(Q)$, that is, $\bar d=\bar d(\mu,m,{\rm tr}(Q))$. As we have seen above, for fixed $\mu$ we can make $\bar d$ arbitrary small if
we choose $m$ sufficiently large and ${\rm tr}(Q)$ sufficiently small. Then we have
\begin{equation*}
 \bar d=\EE 1_A(\omega)\sup_{r\in[-\mu m,0]}Y^m(\theta_{r}\omega)+\EE 1_{A^c}(\omega)\sup_{r\in[-\mu m,0]}Y^m(\theta_{r}\omega)\le \frac{1}{m}+\frac{\varepsilon}{2}<\varepsilon,
\end{equation*}
Indeed, for $\omega \in A$ the path $\theta_r \omega \in B$ for $r\in [-\mu m,0]$.

\end{proof}

\begin{theorem}\label{t3}
Under the assumptions of Lemma \ref{le1}, the sequence $(T_k(\omega))_{k\in \ZZ}$ of stopping times is such that
\begin{equation*}
\liminf_{k\to \infty}\frac{T_{k}(\omega)}{k} \geq D,
\end{equation*}
on a $(\theta_t)_{t\in \RR}$-invariant set of full measure, where
\begin{equation*}
  D=D(\mu,{\rm tr}(Q))= D= \frac{ \mu d}{ 1+d\bar d}\in (0,\mu]
\end{equation*}
such that
\begin{align} \label{eq7}
\lim_{\mu\to 0}D(\mu,{\rm tr}(Q))=0,\qquad  \lim_{{\rm tr}(Q)\to 0}D(\mu,{\rm tr}(Q))=\mu.
\end{align}
\end{theorem}

\begin{proof}

Given $k\in \NN$ we can find an $n=n(k,\omega)\in \NN$ such that
\[
    \sum_{j=0}^{n-1}M_{j}(\omega)<  k\leq \sum_{j=0}^{n}M_{j}(\omega),
\]
where $M_{j}(\omega)$ has been defined in Lemma \ref{l2n}. Then, due to the order of the stopping times  we have
$$T_{k}(\omega) > T_{\sum_{j=0}^{n-1} M_{j}(\omega)}(\omega),$$
and since \[
    \frac{1}{\sum_{j=0}^{n}M_{j}(\omega)}\leq \frac{1}{k} < \frac{1}{\sum_{j=0}^{n-1}M_{j}(\omega)}
\]
we obtain
\begin{align}\label{eq4b}
\begin{split}
 \liminf_{k\to \infty}\frac{T_{k}(\omega)}{k}&\geq \liminf_{n \to \infty}\frac{T_{\sum_{j=0}^{n-1} M_{j}(\omega)}(\omega)}{\sum_{j=0}^{n} M_{j}(\omega)} \\
     &\ge
    \liminf_{n \to\infty}\frac{n \mu}{\sum_{j=0}^{n} N(\theta_{j\mu}\omega)+\sum_{j=0}^n 1_{\{M_j(\cdot)> N(\theta_{j\mu}\cdot)\}}(\omega)}
    \end{split}
\end{align}
where this last inequality is derived from the fact that $T_{\sum_{j=0}^{n-1} M_{j}(\omega)}(\omega)$ is the first stopping time in the interval $(n\mu,(n+1)\mu]$ and Lemma \ref{l2n}.

Now, on account of \eqref{EK} and Lemma \ref{le1}, we obtain
\begin{align*}
 \limsup_{n \to\infty}&\frac{\sum_{j=0}^{n} N(\theta_{j\mu}\omega)+\sum_{j=0}^{n} 1_{\{M_j(\cdot)> N(\theta_{j\mu}\cdot)\}}(\omega)}{n \mu }\\
&=\frac{1}{\mu}  \limsup_{n \to\infty}\frac{\sum_{j=0}^{n} N(\theta_{j\mu}\omega)+\sum_{j=0}^{n} 1_{\{M_j(\cdot)> N(\theta_{j\mu}\cdot)\}}(\omega)}{n+1} \leq \frac{1}{ \mu d}+\frac{1}{\mu} {\bar d},
\end{align*}
on a $(\theta_t)_{t\in \RR}$--invariant set of full measure.
We choose
\begin{equation}\label{D}
  D= \frac{ \mu d}{ 1+d\bar d}
\end{equation}
which has the desired properties described in \eqref{eq7}: the first limit is clear, and for the second, we only need to take into account that, if ${\rm tr}(Q)\to 0$, then $d\to 1 $ and $\bar d\to 0$, hence
\begin{align*}
\lim_{{\rm tr}(Q)\to 0}D(\mu,{\rm tr}(Q))=\mu.
\end{align*}

\end{proof}

\begin{remark}
Now we would like to give a sufficient condition ensuring {\bf(S3)} when considering the fractional Brownian motion with Hurst parameter $H\in (1/2,1)$. To be more precise, remember that
\begin{align*}
\frac{1}{d}&=\mu \EE \bigg( \frac{c_{\alpha,\beta,\beta^\prime} c_{DG}\sup_{r\in [-\mu,0]} \ltn\theta_r  \omega\rtn_{\beta^{\prime\prime},0,\mu}+c_{DF}\mu^{1-\beta^{\prime\prime}}}{c_{DF}\mu}\bigg)^{\frac{1}{\beta^{\prime \prime}}} \\&\leq \mu \EE  \bigg( \frac{c_{\alpha,\beta,\beta^\prime} c_{DG}\ltn  \omega\rtn_{\beta^{\prime\prime},-\mu,\mu}+c_{DF}\mu^{1-\beta^{\prime\prime}}}{c_{DF}\mu}\bigg)^{\frac{1}{\beta^{\prime \prime}}}=\EE \left(\frac{c_{\alpha,\beta,\beta^\prime} c_{DG}\ltn \omega \rtn_{\beta^{\prime \prime},-\mu,\mu}}{c_{DF}\mu^{1-\beta^{\prime \prime}}} +1\right)^{\frac{1}{\beta^{\prime \prime}}},
\end{align*}
and since \[
(x+1)^{\frac{1}{\beta^{\prime \prime}}} \leq 1 + \frac{1}{\beta^{\prime \prime}} x + \frac{1}{2\beta^{\prime \prime}} x^2, \qquad x\geq 0,
\] we have
\begin{align*}
\frac{1}{d} &\leq 1+ \frac{1}{\beta^{\prime \prime}}\frac{c_{\alpha,\beta,\beta^\prime} c_{DG}}{c_{DF}} \mu^{\beta^{\prime \prime}-1} \EE \ltn \omega \rtn_{\beta^{\prime \prime},-\mu,\mu} + \frac{1}{2\beta^{\prime \prime}} \bigg(\frac{c_{\alpha,\beta,\beta^\prime} c_{DG}}{c_{DF}}\bigg)^2\mu^{2(\beta^{\prime \prime}-1)}  \EE \ltn \omega \rtn^2_{\beta^{\prime \prime},-\mu,\mu}.
\end{align*}
Moreover, since for any $q \geq 1$

\[
\EE \ltn \omega \rtn^q_{\beta^{\prime \prime},-\mu,\mu} \leq C_{H,\beta^{\prime \prime},q} [{\rm tr }(Q)]^{q/2} \mu^{(H-\beta^{\prime \prime})q},
\]
then \begin{eqnarray}\label{d-est}
\frac{1}{d} \leq 1+ p_1 \mu^{H-1} + p_2 \mu^{2(H-1)}
\end{eqnarray}
with
$$ p_1:=\frac{1}{\beta^{\prime \prime}}  \frac{c_{\alpha,\beta,\beta^\prime} c_{DG}}{c_{DF}}[{\rm tr }(Q)]^{1/2} C_{H,\beta^{\prime \prime},1}, \qquad p_2 := \frac{1}{2\beta^{\prime \prime}} \bigg(\frac{c_{\alpha,\beta,\beta^\prime} c_{DG}}{c_{DF}}\bigg)^2 [{\rm tr }(Q)] C_{H,\beta^{\prime \prime},2}.$$

Therefore, if we consider a noisy input such that ${\rm tr} (Q) \to 0$, then $p_1,\,p_2 \to 0$, and thus $d \to 1$. On the other hand, see Lemma \ref{le1}, $\bar d\to 0$. As a result, if we perturb a deterministic evolution equation that it is  exponentially stable with a {\it small} fractional Brownian motion, then it turns out that the resulting stochastic evolution equation driven by this fractional Brownian motion keeps  exponentially stable, see Section \ref{s4} below.

\bigskip

In general, note that in view of (\ref{d-est}), since
$$\frac{1}{D}=\frac{1}{\mu}\bigg(\frac{1}{d}+\bar d\bigg),$$
a sufficient condition for {\bf (S3)} is that for all fixed parameters we choose $\mu \in (0,\min \{1, 1/{c_S c_{DF}} \})$ such that
$$  \lambda - e^{\lambda \mu} \left(1+ \frac{p_1}{\mu^{1-H}}+\frac{p_2}{\mu^{2(1-H)}}+\bar d \right)\frac{c_S c_{DF}}{1-c_S c_{DF}\mu} >0. $$
To simplify the problem, we will consider instead the stronger condition
\begin{align}
\lambda - e^{\lambda \mu} \left(1+ \frac{p}{\mu^{q}}\right)\frac{c_S c_{DF}}{1-c_S c_{DF}\mu} >0 \label{simp_cond}
\end{align}
with $p=p_1+p_2$ and $q=2(1-H)$.

Now define
$$ K: (0, 1/{c_S c_{DF}}) \rightarrow (0, \infty), \qquad K(\mu)=e^{\lambda \mu} \left(1+\frac{p}{\mu^{q}} \right)\frac{c_S c_{DF}}{1-c_S c_{DF}\mu} .$$
We have $$K(0^+) = \infty, \qquad K( (1/{c_S c_{DF}})^-) = \infty,$$ and since
\[
\frac{d K}{d\mu} =K(\mu)\Big (\lambda + \frac{c_S c_{DF}}{1-c_S c_{DF}\mu} - \frac{pq}{p\mu + \mu^{q+1}}\Big),
\] the function $K$ is a convex in $(0,1/{c_S c_{DF}} )$ with unique minimum $\mu(\lambda)$ given by
the equation
\[
\lambda =  \frac{1}{\mu(\lambda)} \left( \frac{pq}{p+(\mu(\lambda))^{q}}  -  \frac{c_S c_{DF} \mu(\lambda)}{1-c_S c_{DF} \mu(\lambda)} \right).
\]
Hence, a sufficient condition ensuring {\bf(S3)} is that
$$\lambda-K(\mu(\lambda))>0.$$
\end{remark}


\section{Comparison with the deterministic equation}\label{s4}

In this last section we would like to compare the exponential stability criteria for deterministic equations with the one that we have required (condition {\bf(S3)}) in Theorem \ref{es}.

In the deterministic setting, that corresponds to the choice $G\equiv 0$, for any fixed $\mu >0$ the sequence of stopping times can be constructed simply as $T_0 = 0$ and
\[
T_{n+1} - T_n = \mu,\quad  n \geq 0,
\]
namely
\[
T_{n} = {\mu n},\quad  n \geq 0,
\]
hence {\bf(S2)} is fulfilled:
$$\liminf_{n\to \infty}\frac{T_n}{n}=\liminf_{n\to \infty}\frac{\mu n}{n} =\mu\geq D.$$

Assume that
\begin{equation}\label{det}
\lambda >c_Sc_{DF}
\end{equation}
Then we may find a $0<\mu<1/c_Sc_{DF}$ such that
\begin{equation*}
\lambda- \frac{c_S c_{DF}}{(1-c_S c_{DF} \mu)}e^{{\lambda \mu}} >0.
\end{equation*}
Furthermore, since for any fixed $\mu$ we know that
\begin{equation*}
  \lim_{{\rm tr}(Q)\to 0}D(\mu,{\rm tr}(Q))=\mu,
\end{equation*}
see (\ref{eq7}), then for a small noise or, in other words, for sufficient small ${\rm tr}(Q)>0$ (or $c_{DG}$ small), we have
\begin{equation*}
\lambda- \frac{c_S c_{DF} \mu}{(1-c_S c_{DF} \mu)D}e^{{\lambda \mu}} >0,
\end{equation*}
which is exactly the condition (\ref{rho}) that causes exponential stability.

In our situation, due to the fact that we are considering the $C^\beta_\beta([t,t+\mu];V)$-norm of mild solutions  (i.e.
exponential convergence of $\|u\|_{\beta,\beta,t,t+\mu}$ for small $\mu$), the exponential stability criteria when $G\equiv 0$ is exactly (\ref{det}). As a conclusion, when we perturb a deterministic system
$$du(t)=(Au(t)+F(u(t)))dt, \quad u(0)=u_0$$
with a small fractional Brownian motion, in the sense that its covariance operator $Q$ is such that ${\rm tr}(Q)$ is sufficiently small, then the perturbed stochastic system remains exponentially stable.


\end{document}